%% file: main.tex
\documentclass[letterpaper, 11 pt]{article}
\usepackage{fullpage,amsthm, amsmath, amssymb, amsfonts}
\usepackage{tikz,subfig,standalone}

\usepackage[margin=1in]{geometry}
\usepackage{setspace}
\usepackage{graphicx}
\usepackage{mathtools}

\newtheorem{thm}{Theorem}[section]
\newtheorem{theorem}[thm]{Theorem}

\newtheorem{proposition}[thm]{Proposition}
\newtheorem{lemma}[thm]{Lemma}

\newtheorem{conjecture}[thm]{Conjecture}

\newtheorem{corollary}[thm]{Corollary}

\newtheorem{definition}[thm]{Definition}

\newtheorem{remark}[thm]{Remark}

\newcommand{\newword}[1]{\textbf{\emph{#1}}}

\newcommand{\RR}{\mathbb{R}}
\newcommand{\M}{\mathcal{M}}

\newcommand{\B}{\mathcal{B}}

\newcommand{\union}{\cup}
\DeclarePairedDelimiter\parentheses{\lparen}{\rparen}
\newcommand{\rk}[1]{\operatorname{rk} \parentheses*{#1}}
\newcommand{\drk}[1]{\operatorname{rk}^{*} \parentheses*{#1}}
\newcommand{\NC}[1]{\operatorname{NC} \parentheses*{#1}}
\newcommand{\nbd}[1]{\operatorname{nbd} \parentheses*{#1}}
\newcommand{\cw}[1]{\operatorname{cw} \parentheses*{#1}}
\newcommand{\ccw}[1]{\operatorname{ccw} \parentheses*{#1}}

\def\sample{\min\nolimits_{\Pi \in \NC{[s]}} \nbd{E,\Pi}}
\usepackage{lipsum}

\begin{document}

\title{The facets of the matroid polytope and the independent set polytope of a positroid}

\author{Suho Oh and David Xiang}
\maketitle

\begin{abstract}
A positroid is a special case of a realizable matroid that arose from the study of the totally nonnegative part of the Grassmannian by Postnikov \cite{Postnikov}. In this paper, we study the facets of its matroid polytope and the independent set polytope. This allows one to describe the bases and independent sets directly from the decorated permutation, bypassing the use of the Grassmann necklace. We also describe a criterion for determining whether a given cyclic interval is a flat or not using the decorated permutation, then show how it applies to checking the concordancy of positroids.
\end{abstract}

\section{Introduction}

A positroid comes from taking the nonzero maximal minors in a full-rank $k \times n$ matrix with all maximal minors nonnegative \cite{Postnikov}. Positroids have a number of nice combinatorial properties. In particular, positroids are in bijection with certain interesting classes of combinatorial objects, such as Grassmann necklaces and decorated permutations. Recently, positroids have seen increased applications in physics, with use in the study of scattering amplitudes \cite{arkani} and the study of shallow water waves \cite{kodama}.

A matroid can be described in multiple ways, using bases, independent sets, circuits, the rank function, flats, etc. There have been multiple results on the bases of a positroid: the set of bases can be described nicely from the Grassmann necklace \cite{Oh}, and the (matroid) polytope coming from the bases can be described using the cyclic intervals \cite{ARW},\cite{LP}. Using the rank-function described in \cite{OR}, we will describe all the facets of this polytope using the decorated permutation. This gives a way to describe the bases without using the Grassmann necklace. We also describe the facets of the independent set polytope of the positroid, using the decorated permutation. This gives a way to describe all the independents sets, again without relying on the Grassmann necklace. Lastly, we will show how the results can be applied to checking the concordancy of positroids.

The structure of the paper is as follows. In Section $2$, we go over the background materials needed for this paper, including the basics of matroids, positroids and decorated permutations. In section $3$, we go over the rank function of a positroid. In section $4$ we describe the interval flats and inseparable flats of a positroid. In section $5$, we state our main result. In section 6, we show how our results can be applied to checking concordancy of positroids.

\section{Background materials}

\subsection{Matroids}

In this section we review the basics of matroids that we will need. We refer the reader to \cite{Oxley} for a more in-depth introduction to matroid theory.

\begin{definition}
A \newword{matroid} is a pair $(E,\B)$ consisting of a finite set $E$, called the \newword{ground set} of the matroid, and a nonempty collection of subsets $\B = \B(\M)$ of $E$, called the \newword{bases} of $\M$, which satisfy the \newword{basis exchange axiom}:

If $B_1,B_2 \in \B$ and $b_1 \in B_1 \setminus B_2$, then there exists $b_2 \in B_2 \setminus B_1$ such that $B_1 \setminus \{b_1\} \union \{b_2\} \in \B$. 
\end{definition}

A subset $F \subseteq E$ is called \newword{independent} if it is contained in some basis. All maximal independent sets contained in a given set $A \subseteq E$ have the same size, called the \newword{rank} $\rk{A}$ of $A$. The rank of the matroid $\M$, denoted as $\rk{\M}$, is given by $\rk{E}$. An element $e \in E$ is a \newword{loop} if it is not contained in any basis. An element $e \in E$ is a \newword{coloop} if it is contained in all bases. A matroid $\M$ is \newword{loopless} if it does not contain any loops. The \newword{dual} of $\M$ is a matroid $\M^{*} = (E,\B')$ where $\B' = \{E \setminus B | B \in \B(\M)\}$.

\begin{remark}
Due to the reason explained in Remark~\ref{rem:nofixed}, in this paper, we will be sticking to matroids that have no loops nor coloops.
\end{remark}

 The \newword{closure} of a set $A$ is denoted as $\bar{A}$, and stands for the biggest set that contains $A$ and has the same rank. A set is a \newword{flat} if its closure is same as itself. A set $E$ is called \newword{separable} in a matroid if one can partition $E$ into $E_1$ and $E_2$ such that $\rk{E} = \rk{E_1} + \rk{E_2}$.

\begin{remark}
In this paper, we will always use $[n] := \{1,\ldots,n\}$ as our ground set, reserving the usage of $E$ for subsets of the ground set we analyze. A matroid of rank $d$ will have bases in the set ${[n] \choose d}$ which stands for all cardinality $d$-subsets of $[n]$. Also for ease of reading we will use $B \in \M$ to denote $B \in \B(\M)$.
\end{remark}



\begin{definition}
Given a matroid $\M = ([n],\B)$, the (basis) matroid polytope $\Gamma_{\M}$ of $\M$ is the convex hull of the indicator vectors of the bases of $\M$:
$$\Gamma_{\M} = convex\{e_B|B \in \B\} \subset \RR^n,$$
where $e_B := \sum_{i \in B} e_i$ and $\{e_1,\ldots,e_n\}$ is the standard basis of $\RR^n$.
\end{definition}

\begin{definition}
Given a matroid $\M = ([n],\B)$, the independent set polytope $P_{\M}$ of $\M$ is the convex hull of the indicator vectors of the independent sets of $\M$:
$$P_{\M} = convex\{e_I|I \subset B \in \B\} \subset \RR^n,$$
where $e_I := \sum_{i \in I} e_i$ and $\{e_1,\ldots,e_n\}$ is the standard basis of $\RR^n$.
\end{definition}

There is a nice description for the facets of these polytopes.

\begin{theorem}[Proposition 2.6. of \cite{Feichtner2005}, Definition 2.3. of \cite{ku_2019}]
\label{thm:matpoly}
The following is a system of inequalities (with one equality) for the matroid polytope of $\M = ([n],\B)$:
\begin{itemize}
\item $x_e \geq 0$, $e \in [n]$,
\item $\sum_e x_e = d$,
\item $x_F := \sum_{e \in F} x_e \leq \rk{F}$, $F$ is a \newword{flacet} : $\emptyset \subsetneq F \subsetneq [n]$ is a flat of $\M$ where $F$ is inseparable in $\M$ and $F^c$ is inseparable in the dual of $\M$.
\end{itemize}
\end{theorem}

\begin{theorem}[Theorem 40.5. of \cite{Schrijver}]
\label{thm:indpoly}
If $\M = ([n],\B)$ is loopless, the following is a minimal system of inequalities for the independent set polytope of $\M$:
\begin{itemize}
\item $x_e \geq 0$, $e \in [n]$,
\item $x_F := \sum_{e \in F} x_e \leq \rk{F}$, $F$ is a nonempty inseparable flat of $\M$,
\end{itemize}
\end{theorem}

In this paper, we will show a method to read off the $F$'s and their ranks for both of those polytopes when dealing with a positroid, directly from the associated decorated permutation.

\subsection{Positroids}

In this section we go over the basics of positroids. Positroids were originally defined in \cite{Postnikov}. Positroids are matroids whose bases are the column sets corresponding to nonzero maximal minors in a matrix such that all maximal minors are nonnegative. For example, the following matrix has nonnegative maximal minors:
\[ A = \left( \begin{array}{cccc}
1 & 0 & -3 & -1 \\
0 & 1 & 4 & 0\\
\end{array} \right).\]

The nonzero maximal minors come from the column sets $\{1,2\}, \{1,3\},\{2,3\},\{2,4\},\{3,4\}$. This collection forms a positroid. Positroids are in bijection with decorated permutations:

\begin{definition}
A \newword{decorated permutation} of the set $[n]$ is a bijection $\pi$ of $[n]$ whose fixed points are colored either white or black. 
\end{definition}


\begin{figure}
    \centering
    \input{tikzs/permutation.tex}
    \caption{A decorated permutation $\pi = [2,8,6,7,9,4,5,14,13,3,10,11,1,12].$}
		\label{fig:permutation}
\end{figure}

Take a look at the decorated permutation (since it has no fixed points, it is the usual permutation) in Figure~\ref{fig:permutation}. It is the permutation $[2,8,6,7,9,4,5,14,13,3,10,11,1,12]$ under the usual one-line notation.

Given $a,b \in [n]$, we define the (cyclic) interval $[a,b]$ to be the set $\{x| x \leq_a b\}$, where the \newword{cyclically shifted order} $<_i$ on the set $[n]$ is the following total order:
$$i <_i i+1 <_i \cdots <_i n <_i 1 <_i \cdots <_i i-1. $$

\begin{remark}
When we are dealing with positroids, we will always envision the ground set $[n]$ to be drawn on a circle. We will say that $a_1,\ldots,a_t \in [n]$ are \newword{cyclically ordered} if there exists some $i \in [n]$ such that $a_1 <_i \cdots <_i a_t$. 
Cyclic intervals play an important role in the structure of a positroid \cite{Knutson}. 
\end{remark}


\begin{theorem}[\cite{ARW}, \cite{LP}]
\label{thm:posipoly}
A matroid $\M$ of rank $d$ on $[n]$ is a positroid if and only if its matroid polytope $\Gamma_{\M}$ can be described by the equality $x_1 + \cdots + x_n = d$ and inequalities of form
$$\sum_{l \in [a,b]} x_l \leq \rk{[a,b]}, \text{ with } i,j \in [n].$$
\end{theorem}

Once a positroid is fixed, a lot of inequalities in the above theorem do not actually correspond to facets of the matroid polytope. Later in the paper, we will show how to obtain the cyclic intervals that are flats (hence one can ignore all other intervals when studying the facets), directly from the decorated permutation.

\begin{remark}
\label{rem:nofixed}
It is enough to consider the positroid $\mathcal{M}'$ obtained by deleting any loops and coloops in order to study the structural properties of $\mathcal{M}$. For the remainder of this paper, we will assume positroids are loopless and coloopless. This means that the associated decorated permutations have no fixed points.
\end{remark}

We end the subsection with a way to count the rank of a positroid using its decorated permutation:
\begin{lemma}[Proposition 16.4 of \cite{Postnikov}, Definition 4.5 of \cite{ARW}]
\label{lem:rkexcedance}
The rank of a positroid $\M$ with an associated decorated permutation $\pi$ is obtained by counting the $i$-excedances of $\pi$ for any $i \in [n]$: the set of $x \in [n]$ such that $\pi^{-1}(x) >_i x$.
\end{lemma}

For example, the positroid associated to the decorated permutation $\pi$ of Figure~\ref{fig:permutation} has the following $1$-excedances: $\{4,5,3,10,11,1,12\}$. Therefore the rank of the positroid is $7$.

\subsection{Dual matroids}
Given a matroid $\M = ([n], \B)$, its \newword{dual} matroid $\M^*$ can be described in terms of the bases: take the collection of $[n] \setminus B$ for each $B \in \B(\M)$ to form $\B(\M^*)$.

Then we have a nice relation between the rank function of the original matroid and the dual matroid:

\begin{lemma}[Proposition 2.1.9 of \cite{Oxley}]
\label{lem:rkdual}
Let $\M$ be a matroid. We use $\drk{}$ to denote the rank function of the dual matroid $\M^*$. We have the following relation between the rank functions of $\M$ and its dual:
$$\drk{E} = \rk{E^c} + |E| - \rk{\M}.$$
\end{lemma}

Notice that plugging in $E = [n]$ above gives us $\rk{\M^*} = n - \rk{\M}$. Recall that positroids are indexed with a permutation. The dual matroid of a positroid turns out to be a positroid as well, with a nice relation between the associated permutations:

\begin{theorem}[Proposition 3.5 of \cite{ARW}, Corollary 13 of \cite{posifpsac}]
\label{thm:dualposi}
Let $\M$ be a positroid. Then its dual matroid $\M^*$ is also a positroid. Moreover, if $\M$ was indexed by the decorated permutation $\pi$ with no fixed points, then $\M^*$ is indexed by the permutation $\pi^{-1}$.
\end{theorem}

\section{Rank function of a positroid and non-crossing partitions}
In this section we review the result of \cite{OR}. The rank of a cyclic interval of a positroid is obtained by counting the number of counter-clockwise arrows of the permutation contained outside the interval. Rank of a set that consists of unions of cyclic intervals can be obtained in a similar manner, with the help of non-crossing partitions.

\begin{definition}
Let $\Pi = T_1 \sqcup \cdots \sqcup T_p$ be a partition of $[s]$ into pairwise disjoint non-empty subsets. We say that $\Pi$ is a \newword{non-crossing partition} if there are no cyclically ordered $a,b,c,d$ such that $a,c \in T_i$ and $b,d \in T_j$ for some $i \not = j$. We will call the $T_i$'s the \newword{blocks} of the partition.
\end{definition}


We will call an interval of form $[x,\pi(x)]$ a \newword{CW-arrow} of $\pi$, and an interval of form $[x,\pi^{-1}(x)]$ a \newword{CCW-arrow} of $\pi$ (each standing for clockwise and counterclockwise). Given any set $E \subseteq [n]$, we use $\cw{E}$ to denote the number of CW-arrows contained in $E$ (the entire interval has to be contained in $E$). Similarly, we will use $\ccw{E}$ for the number of CCW-arrows contained in $E$.

Let $E$ be any subset of the ground set. Then the \newword{natural rank bound} of $E$, written as $\nbd{E}$, is given by the rank of $\M$ minus the number of counter-clockwise arrows contained in the complement of $E$. 

Any set $E \subseteq [n]$ can be written as a disjoint union of cyclic intervals, $E = [a_1,b_1] \union \cdots \union [a_s,b_s]$ where $a_1,b_1,\ldots,a_s,b_s$ are cyclically ordered. Let $\Pi$ be an arbitrary non-crossing partition of $[s]$ with $T_1,\ldots,T_p$ as its parts. We define $E|_{T_i}$ as the subset of $E$ obtained by taking only the intervals indexed by elements of $T_i$. For example, $E|_{\{1,3\}}$ would stand for $E_1 \cup E_3 = [a_1,b_1] \cup [a_3,b_3]$.  We define $\nbd{E, \Pi}$ to be $\nbd{E|_{T_1}} + \cdots + \nbd{E|_{T_p}}$. 

\begin{theorem}[\cite{OR}]
\label{thm:rk}
Let $E = [a_1,b_1] \union \cdots \union [a_s,b_s]$ be a disjoint union of $s$ cyclic intervals, where $a_1,b_1,a_2,b_2,\ldots,a_s,b_s$ are cyclically ordered. We have $\rk{E} = \sample$, where $\NC{[s]}$ denotes the set of all non-crossing partitions of $[s]$.
\end{theorem}

\begin{figure}
\begin{center}
    \input{tikzs/nbdex.tex}
    \caption{Information needed to compute $\nbd{E, \{\{1\},\{2\}\}}$ in the process of computing the rank of $E = [1,3] \union [8,10]$.}
		\label{fig:interval13810rk}

\end{center}
\end{figure}

For example, take a look at Figure~\ref{fig:interval13810rk} (the positroid is the one associated to Figure~\ref{fig:permutation}). When we take $E$ to be a cyclic interval, this has only one block and hence the only partition to consider is the trivial partition of $[1]$. If we take $E = [1,3]$, we have $\rk{E} = \nbd{E,[1]} = \nbd{E} = 7-5 = 2$, since the rank of the matroid is $7$ and there are five CCW-arrows outside $[1,3]$ as we can see from Figure~\ref{fig:interval13810rk}.

In order to compute the rank of $E = [1,3] \union [8,10]$ we need to find $\nbd{E, \{\{1\},\{2\}\}}$ and $\nbd{E,\{\{1,2\}\}}$: here $1$ stands for the first block $[1,3]$ and $2$ stands for the second block $[8,10]$. To compute $\nbd{E,\{\{1\},\{2\}\}} = \nbd{E_1} + \nbd{E_2} = \nbd{[1,3]} + \nbd{[8,10]}$, we look at the left figure of Figure~\ref{fig:interval13810rk}. The interval $[8,10]$ has $4$ CCW-arrows contained in its complement. Hence $\nbd{[8,10]} = 7-4 = 3$. Similarly, the interval $[1,3]$ has $5$ CCW-arrows contained in its complement, hence we get $\nbd{[1,3]} = 7-5 = 2$. From this we get that $\nbd{E,\{\{1\},\{2\}\}} = 3+2 = 5$. Now to compute $\nbd{E,\{\{1,2\}\}}$, we count how many CCW-arrows there are in the complement of $E$. We have $4$ as we can see from Figure~\ref{fig:interval13810}, so we get $\nbd{E,\{\{1,2\}\}} = 7-4 = 3$. To get the rank of $E$, we take the minimum of the two values $\nbd{E, \{\{1\},\{2\}\}} = 5 $ and $\nbd{E,\{\{1,2\}\}} = 3$ to end up with $\rk{E} = 3$.

\begin{corollary}
\label{cor:insepnbd}
Let $E = [a_1,b_1] \union \cdots \union [a_s,b_s]$ be a disjoint union of $s$ cyclic intervals, where $a_1,b_1,a_2,b_2,\ldots,a_s,b_s$ are cyclically ordered. If $\rk{E} = \nbd{E,\Pi}$ where $\Pi$ is not $[s]$, the trivial partition, then $E$ is separable. 
\end{corollary}
\begin{proof}
From the submodularity of the rank function, we have $\rk{A} + \rk{B} \geq \rk{A\cup B}$ when $A$ and $B$ are disjoint \cite{Oxley}. Let the blocks we get by partitioning $E$ according to $\Pi$ be $E_1,\ldots,E_p$. Here since $\Pi$ is not the trivial partition, we have $p \geq 2$. We get 
$$\rk{E_1} + \dots + \rk{E_p} \geq \rk{E} = \nbd{E,\Pi} = \nbd{E_1} + \dots + \nbd{E_p} \geq \rk{E_1} + \dots \rk{E_p},$$
where the first inequality comes from submodularity of the rank function and the second inequality comes from Theorem~\ref{thm:rk}. This implies that $E$ is separable.
\end{proof}

From this corollary, we can see that whenever $E$ is an inseparable set, its rank can be obtained very easily from the decorated permutation: since we have $\rk{E} = \nbd{E}$, all we need to do is count the number of CCW-arrows outside $E$ and negate it from the rank of $\M$. 

Finally, in the case $E$ is a cyclic interval, we can compute the rank using $cw(E)$ instead:
\begin{proposition}
\label{prop:usingcw}
Let $E=[a,b]$ be a cyclic interval. Then $\rk{E} = |E| - \cw{E}$.
\end{proposition}
\begin{proof}
Interpret $\rk{\M}$ as $a$-excedances of $\pi$ from Lemma~\ref{lem:rkexcedance}. Counting the $a$-excedances of $\pi$ is same as counting the CCW-arrows in $[a,a-1]$. From $\rk{E} = \nbd{E} = \rk{\M} - \ccw{E^c}$, the right hand side can be interpreted as the number of CCW-arrows of $[a,a-1]$ that lands somewhere among $E$. This can be counted by starting from total number of entries in $E$ then negating the CW-arrows in $E$.
\end{proof}

For example, again take a look at our running example in Figure~\ref{fig:permutation}. The interval $[1,3]$ has exactly one CW-arrow inside, so $\rk{[1,3]} = 3 - 1 = 2$. Similarly, the interval $[8,10]$ has no CW-arrows inside, so $\rk{[8,10]} = 3 - 0 = 3$.

\section{Interval flats and Inseparable flats}
We will show how to obtain the interval flats and inseparable flats directly from the decorated permutation. Recall that we are sticking to positroids that have no loops nor coloops as in Remark~\ref{rem:nofixed}. Using Theorem~\ref{thm:rk}, we get the following result:
\begin{theorem}
\label{thm:gencover}
Let $\M$ be a positroid with associated decorated permutation $\pi$ and let $E \subseteq [n]$ be an inseparable set. Then $E$ is a flat of $\M$ if and only if each element of $E^c$ is contained in some counter-clockwise arrow of $\pi$ contained in $E^c$. If this happens, we say that \newword{$E^c$ is covered by CCW-arrows}.
\end{theorem}

\begin{proof}
Let $E$ be the disjoint union of $s$ cyclic intervals. Since $E$ is inseparable, from Corollary~\ref{cor:insepnbd}, we have $\rk{E} = \nbd{E} < \nbd{E,\Pi}$ where $\Pi$ is any nontrivial non-crossing partition of $[s]$. Let $x$ be an arbitrary element of $[n] \setminus E$ and set $Ex$ to denote $E \union \{x\}$. We can think of $Ex$ having $s+1$ cyclic intervals (the interval consisting of the lone element $x$ might be adjacent to another cyclic interval, but it does not matter. Index that lone interval as the $s+1$-th interval). Let $\Pi'$ be an arbitrary non-crossing partition of $[s+1]$ and let $\Pi$ be obtained from $\Pi'$ by deleting the element $s+1$. Then we have $\nbd{E} < \nbd{E,\Pi} \leq \nbd{Ex,\Pi'}$ for any $\Pi'$ that isn't $\{\{1,\ldots,s\},\{s+1\}\}$ or trivial. In the case $\Pi' = \{\{1,\ldots,s\},\{s+1\}\}$, from the assumption that $\M$ is loopless, we get $\nbd{Ex,\{\{1,\ldots,s\},\{s+1\}\}} = \nbd{E} + \nbd{x} > \nbd{E}$ (since $x$ is not a loop, we have $\nbd{x} \geq \rk{x} \geq 1$). 

Hence aside from the case $\Pi'$ is trivial, we know that $\nbd{E} < \nbd{Ex,\Pi'}$. From this we can conclude that $\rk{E} < \rk{Ex}$ if and only if $\nbd{E} < \nbd{Ex}$. Now notice that $\nbd{E} < \nbd{Ex}$ if and only if there is a CCW-arrow outside $E$ that contains $x$. Finally recall that $E$ is a flat if and only if $\rk{E} < \rk{Ex}$ for all $x \in [n] \setminus E$. Combining these two statements we get the desired result.
\end{proof}

In the proof of the above theorem, what we actually require on $E$ is to satisfy $\rk{E} = \nbd{E}$ instead of being inseparable. In the case $E$ is a cyclic interval, we have $\rk{E} = \nbd{E}$ even when $E$ is separable, since $E$ only has one block and hence only possible option for $\Pi$ is the trivial partition of $[1]$. So as a corollary, we get:

\begin{theorem}
\label{thm:intervalcover}
Let $\M$ be a positroid and let $E \subseteq [n]$ be a cyclic interval. Then $E$ is a flat of $\M$ if and only if $E^c$ is covered by CCW-arrows.
\end{theorem}

For example, take a look at Figure~\ref{fig:interval110}. The complement of the interval $[1,10]$ is covered by CCW-arrows disjoint from $[1,10]$. So this is a flat. On the other hand, in the complement of the interval $[1,3]$, the elements $8$ and $9$ in particular are not covered by CCW-arrows outside $[1,3]$. So $[1,3]$ is not a flat (its closure is $[1,3] \union [8,9]$).

\begin{remark}
Beware that we only care about integers of an interval when we discuss the covering of an interval. For example, if there are two CCW-arrows $[7,9]$ and $[10,11]$, we say that $[7,11]$ is covered by CCW-arrows even if there is no CCW-arrow covering the region between $9$ and $10$.
\end{remark}

\begin{figure}
    \centering
    \input{tikzs/flatex.tex} 
    \caption{The interval $[1,10]$ is a flat. The interval $[1,3]$ is not.}
		\label{fig:interval110}   
\end{figure}

\begin{remark}
The study of cyclic intervals that are flats was motivated from the \newword{essential intervals} studied in \cite{Knutson}. We would like to point out that the set of essential intervals and the set of interval flats are incomparable: there are essential intervals that are not flats and there are interval flats that are not essential.
\end{remark}

Although not used for our main result, it is worth noting that arbitrary intersection of interval flats can be described using a similar criterion.

\begin{corollary}
\label{cor:generalcover}
Let $E$ be an arbitrary subset of $[n]$. Then $E$ is the intersection of interval flats if and only if $E^c$ is covered by CCW-arrows.
\end{corollary}
\begin{proof}
Let $E = [a_1,b_1] \union \cdots \union [a_s,b_s]$ be a disjoint union of $s$ cyclic intervals, where $a_1,b_1,a_2,b_2,\ldots,a_s,b_s$ are cyclically ordered. Notice that for each $i$, the interval $[a_i,b_{i-1}]$ is a flat since $(b_{i-1},a_i)$ is covered by CCW-arrows. To get $E$, simply take the intersection of all $[a_i,b_{i-1}]$'s.

\end{proof}

\begin{figure}[h]
    \centering
    \include{tikzs/flatinter}
    \caption{The set $[1,3] \union [8,10]$ is the intersection of flats $[1,10]$ and $[8,3]$.}
		\label{fig:interval13810}

\end{figure}

For example, take a look at Figure~\ref{fig:interval13810}. The complement of $[1,3] \union [8,10]$ consists of the intervals $(3,8)$ and $(10,1)$. And each of those intervals is covered by CCW-arrows that do not intersect $[1,3] \union [8,10]$. Hence $[1,3] \union [8,10]$ is the intersection of interval flats. In particular, it is the intersection of flats $[1,10]$ and $[8,3]$.

\section{Facets of polytopes}
In this section, we state our main result using the tools from the previous section.

\begin{theorem}
Let $\M$ be a positroid of rank $d$ on $[n]$ and $\pi$ be its associated decorated permutation. Its matroid polytope $\Gamma_{\M}$ can be described by the inequalities $x_i \geq 0$ for all $i \in [n]$, the equality $x_1 + \cdots + x_n = d$ and inequalities of form
$$\sum_{l \in E} x_l \leq d - \ccw{E^c}$$
where $E$ is a cyclic interval whose complement is covered by CCW-arrows of $\pi$ and $\ccw{E^c}$ counts the number of CCW-arrows in $E^c$.
\end{theorem}

\begin{proof}
From combining Theorem~\ref{thm:matpoly} and Theorem~\ref{thm:intervalcover}, it is enough to show that the  flacets of a positroid are all interval flats. Assume for sake of contradiction that we have some set $E \subset [n]$ that is not a cyclic interval, but is a flacet. Let $E = [a_1,b_1] \union \cdots \union [a_s,b_s]$ be a disjoint union of $s \geq 2$ cyclic intervals, where $a_1,b_1,a_2,b_2,\ldots,a_s,b_s$ are cyclically ordered. Since $E$ is a flacet, from Corollary~\ref{cor:insepnbd}, we have $\rk{E} = \nbd{E}$.

Therefore $\rk{E} = \rk{\M} - \sum_i \ccw{(b_i,a_{i+1})}$. Now let us try to convert this to information to that of $\M^{*}$, which is a positroid indexed by $\pi^{-1}$ thanks to Theorem~\ref{thm:dualposi}. We are going to add $|E^c| - \rk{\M}$ to both sides. Left hand side becomes $\drk{E^c}$ from Lemma~\ref{lem:rkdual}. Right hand side becomes $|E^c| - \sum_i \ccw{(b_i,a_{i+1})} = \sum_i \drk{(b_i,a_{i+1})}$ thanks to Proposition~\ref{prop:usingcw}. This means that $E^c$ is separable in $\M^{*}$ which leads to a contradiction.
\end{proof}

For example, take a look at the positroid coming from the decorated permutation of Figure~\ref{fig:permutation}. Recall that $[1,10]$ is a flat and $[1,3]$ is not. Hence $x_1 + \cdots x_{10} =  7 - 2$ (there are $2$ counter-clockwise arrows contained outside $[1,10]$) is one of the defining hyperplanes of the matroid polytope. And $x_1 + \cdots + x_3 = t$ for any number $t$ is not.

We also get an analogous result for independent sets:

\begin{theorem}
Let $\M$ be a positroid of rank $d$ on $[n]$ and $\pi$ be its associated decorated permutation. Its independent set polytope $\Gamma_{\M}$ can be described by inequalities $x_i \geq 0$ for each $i \in [n]$ and inequalities of form
$$\sum_{l \in E} x_l \leq d - \ccw{E^c}$$
where $E$ is a subset of $[n]$ whose complement is covered by CCW-arrows of $\pi$ and $\ccw{E^c}$ counts the number of CCW-arrows in $E^c$.
\end{theorem}

\begin{proof}
This follows from Theorem~\ref{thm:indpoly} with Theorem~\ref{thm:gencover}.
\end{proof}

For example, again look at our running example of Figure~\ref{fig:permutation}. For the independent set polytope, aside from the interval flats, we also have to consider ones that are not intervals. One of the inseparable flats was given by $[1,3] \union [8,10]$ from Figure~\ref{fig:interval13810}, so the corresponding facet of the independent set polytope of the positroid is given by $x_1 + x_2 + x_3 + x_8 + x_9 + x_{10} = 3$, since $\rk{[1,3] \union [8,10]} = 7-4$.

\section{Application to concordancy check}
In this section we look at how our result, mainly Theorem~\ref{thm:intervalcover}, can be applied to checking whether two positroids are \newword{concordant} or not: two matroids of different rank are concordant to each other if they can be realized using the same matrix (one using all the rows, the other cutting off few rows from the bottom). The condition for a positroid to be concordant with a uniform matroid was studied in \cite{quotuni}.

It is known that concordancy can be described using flats:

\begin{theorem}
\label{thm:concordflat}
[Proposition 7.3.6 of \cite{Oxley}]
Two matroids $\M_1$ and $\M_2$ are concordant if and only if every flat of $\M_1$ is a flat of $\M_2$.
\end{theorem}

From this well known result and our result on when exactly the cyclic intervals of a positroid is a flat or not, we get the following proposition:

\begin{proposition}
	Let us have two positroids of rank $a$ and $b$ ($a < b$) each indexed by permutations
$\pi$ and $\mu$ that are concordant. Then every CCW-arrow of $\mu$ is covered by CCW-arrows of $\pi$.
\end{proposition}
\begin{proof}
Each CCW-arrow $[a,b]$ of $\mu$ gives a flat $(b,a)$ in the corresponding positroid. Since this must again be a flat in the positroid indexed by $\pi$ thanks to Theorem~\ref{thm:concordflat}, the interval $[a,b]$ has to be covered by CCW-arrows of $\pi$. 

\end{proof}

An example is shown in Figure~\ref{fig:concordancy}. We have (decorated) permutations $\pi = 35476821$ and $\mu = 46782351$. Each CCW-arrow of $\mu$ is covered by CCW-arrows of $\pi$. For example, take the CCW-arrow $(4,8)$ in $\mu$. This is covered by $(4,7)$ and $(6,8)$ which are CCW-arrows in $\pi$.

From this we can ask if the converse is true as well:

\begin{conjecture}
	Let us have two positroids of rank $a$ and $b$ ($a < b$) each indexed by permutations
$\pi$ and $\mu$. The two positroids are concordant if and only if every CCW-arrow of $\mu$ is covered by CCW-arrows of $\pi$.
\end{conjecture}

\section*{Acknowledgement}
The authors would like to thank the organizers of the $2016$ and $2017$ summer mathcamp hosted at Texas State University, where this project started. The authors would also like to thank Lillian Bu, Wini Taylor-Williams, Andrew Lu, Claire Zhou and Brandon Chen for useful discussions. We would also like to thank the anonymous referee for suggestions on improving the paper.

\begin{figure}
    \centering
    \input{tikzs/concord.tex}
    \caption{Each and every CCW-arrow of $\mu = [4,6,7,8,2,3,5,1]$ (right figure) is covered by CCW-arrows of $\pi = [3,5,4,7,6,8,2,1]$ (left figure).}
		\label{fig:concordancy}
\end{figure}

\bibliographystyle{plain}    
\bibliography{bibliography}

\end{document}

%% file: tikzs/permutation.tex
    
   \begin{tikzpicture}
\draw[black] (0:4) arc (0:360:40mm);
    \foreach \phi in {1,...,14}{
    \node[circle,draw,text=black,fill=white] (\phi) at (110+360/14 * -\phi:4cm) {$\phi$};
      }
    \draw [thick, -latex] (1) to [bend right] (2);
    \draw [thick, -latex] (2) to (8);
    \draw [thick, -latex] (3) to (6);
    \draw [thick, -latex] (4) to [bend right] (7);
    \draw [thick, -latex] (5) to [bend right] (9);
    \draw [thick, -latex] (6) to [bend left] (4);
    \draw [thick, -latex] (7) to [bend left] (5);
    \draw [thick, -latex] (8) to (14);
    \draw [thick, -latex] (9) to (13);
    \draw [thick, -latex] (10) to (3);
    \draw [thick, -latex] (11) to [bend left] (10);
    \draw [thick, -latex] (12) to [bend left] (11);
    \draw [thick, -latex] (13) to [bend right] (1);
    \draw [thick, -latex] (14) to [bend left] (12);

   \end{tikzpicture}

%% file: tikzs/nbdex.tex
   \begin{tikzpicture}
\draw[black] (0:3) arc (0:360:30mm);

    \foreach \phi in {1,...,7}{
    \node[circle,draw,text=black,fill=white] (\phi) at (110+360/14 * -\phi:3cm) {$\phi$};
      }
    \foreach \phi in {8,...,10}{
    \node[circle,draw,text=black,fill=red] (\phi) at (110+360/14 * -\phi:3cm) {$\phi$};
      }    
    \foreach \phi in {11,...,14}{
    \node[circle,draw,text=black,fill=white] (\phi) at (110+360/14 * -\phi:3cm) {$\phi$};
      }
    \draw [thick, -latex] (6) to [bend left] (4);
    \draw [thick, -latex] (7) to [bend left] (5);
    \draw [thick, -latex] (12) to [bend left] (11);
    \draw [thick, -latex] (14) to [bend left] (12);

   \end{tikzpicture}
   \begin{tikzpicture}
\draw[black] (0:3) arc (0:360:30mm);

    \foreach \phi in {1,...,3}{
    \node[circle,draw,text=black,fill=red] (\phi) at (110+360/14 * -\phi:3cm) {$\phi$};
      }
    \foreach \phi in {4,...,14}{
    \node[circle,draw,text=black,fill=white] (\phi) at (110+360/14 * -\phi:3cm) {$\phi$};
      }
    \draw [thick, -latex] (6) to [bend left] (4);
    \draw [thick, -latex] (7) to [bend left] (5);
    \draw [thick, -latex] (11) to [bend left] (10);
    \draw [thick, -latex] (12) to [bend left] (11);
    \draw [thick, -latex] (14) to [bend left] (12);

   \end{tikzpicture}

%% file: tikzs/flatex.tex
   \begin{tikzpicture}
\draw[black] (0:3) arc (0:360:30mm);
    \foreach \phi in {1,...,10}{
    \node[circle,draw,text=black,fill=red] (\phi) at (110+360/14 * -\phi:3cm) {$\phi$};
      }    
    \foreach \phi in {11,...,14}{
    \node[circle,draw,text=black,fill=white] (\phi) at (110+360/14 * -\phi:3cm) {$\phi$};
      }
    \draw [thick, -latex] (12) to [bend left] (11);
    \draw [thick, -latex] (14) to [bend left] (12);

   \end{tikzpicture}
   \begin{tikzpicture}
\draw[black] (0:3) arc (0:360:30mm);

    \foreach \phi in {1,...,3}{
    \node[circle,draw,text=black,fill=red] (\phi) at (110+360/14 * -\phi:3cm) {$\phi$};
      }
    \foreach \phi in {4,...,14}{
    \node[circle,draw,text=black,fill=white] (\phi) at (110+360/14 * -\phi:3cm) {$\phi$};
      }
    \draw [thick, -latex] (6) to [bend left] (4);
    \draw [thick, -latex] (7) to [bend left] (5);
    \draw [thick, -latex] (11) to [bend left] (10);
    \draw [thick, -latex] (12) to [bend left] (11);
    \draw [thick, -latex] (14) to [bend left] (12);

   \end{tikzpicture}

%% file: tikzs/flatinter.tex
   \begin{tikzpicture}
    \draw[black] (0:3) arc (0:360:30mm);
    \foreach \phi in {1,...,3}{
    \node[circle,draw,text=black,fill=red] (\phi) at (110+360/14 * -\phi:3cm) {$\phi$};
      }
    \foreach \phi in {8,...,10}{
    \node[circle,draw,text=black,fill=red] (\phi) at (110+360/14 * -\phi:3cm) {$\phi$};
      }          
    \foreach \phi in {4,...,7}{
    \node[circle,draw,text=black,fill=white] (\phi) at (110+360/14 * -\phi:3cm) {$\phi$};
      }
    \foreach \phi in {11,...,14}{
    \node[circle,draw,text=black,fill=white] (\phi) at (110+360/14 * -\phi:3cm) {$\phi$};
      }      
    \draw [thick, -latex] (12) to [bend left] (11);
    \draw [thick, -latex] (14) to [bend left] (12);
    \draw [thick, -latex] (6) to [bend left] (4);
    \draw [thick, -latex] (7) to [bend left] (5);
   \end{tikzpicture}

%% file: tikzs/concord.tex
   \begin{tikzpicture}
\draw[black] (0:3) arc (0:360:3cm);
    \foreach \phi in {1,...,8}{
    \node[circle,draw,text=black,fill=white] (\phi) at (110+360/8 * -\phi:3cm) {$\phi$};
      }
    \draw [thick, -latex] (1) to (3);
    \draw [thick, -latex] (2) to (5);
    \draw [thick, -latex] (3) to (4);
    \draw [thick, -latex,blue] (4) to (7);
    \draw [thick, -latex] (5) to (6);
    \draw [thick, -latex,blue] (6) to (8);
    \draw [thick, -latex] (7) to (2);
    \draw [thick, -latex] (8) to (1);
   \end{tikzpicture}
\hspace{1cm}
\begin{tikzpicture}
\draw[black] (0:3) arc (0:360:3cm);
    \foreach \phi in {1,...,8}{
    \node[circle,draw,text=black,fill=white] (\phi) at (110+360/8 * -\phi:3cm) {$\phi$};
      }
    \draw [thick, -latex] (1) to (4);
    \draw [thick, -latex] (2) to (6);
    \draw [thick, -latex] (3) to (7);
    \draw [thick, -latex,blue] (4) to (8);
    \draw [thick, -latex] (5) to (2);
    \draw [thick, -latex] (6) to (3);
    \draw [thick, -latex] (7) to (5);
    \draw [thick, -latex] (8) to (1);
   \end{tikzpicture}